\documentclass[11pt, a4paper]{article}

\usepackage{fancyhdr, paralist}
\fancyheadoffset{ 0.5 cm}
\fancyfootoffset{ 0.5 cm}

\usepackage{titlesec}
\titlelabel{\thetitle.\quad}

\usepackage{times}
\usepackage{geometry}
\usepackage{graphicx}
\usepackage{amssymb, amscd}
\usepackage{amsmath}
\makeatletter
\@addtoreset{equation}{section}
\makeatother
\numberwithin{equation}{subsection}
\usepackage{amsthm} 
\usepackage{mathrsfs}
\usepackage{epstopdf}
\usepackage{enumerate}
\usepackage{url}
\usepackage{graphicx}
\usepackage{mathtools}
\usepackage{tikz}							
\usepackage{setspace}
\usepackage{amsfonts, amsmath, wasysym}
\usepackage{stackrel}

\usepackage{hyperref}
\usepackage{cite}

\usepackage[draft]{optional}

\usepackage{color}

\usepackage{xcolor}
\newtagform{red}{\color{red}(}{)}

\usepackage{mathtools}

\usepackage{amsmath}
\makeatletter
\@addtoreset{equation}{section}
\makeatother
\numberwithin{equation}{section}

\theoremstyle{plain}
\newtheorem{theorem}{Theorem}[section]
\newtheorem{lemma}[theorem]{Lemma}
\newtheorem{proposition}[theorem]{Proposition}

\newtheorem{conj}[theorem]{Conjecture}

\theoremstyle{definition}
\newtheorem{definition}[theorem]{Definition}

\theoremstyle{remark}

\newtheorem{remark}[theorem]{\bf{Remark}}

\newcommand{\m}{\ensuremath{{\cal M}}}

\newcommand{\cd}{\ensuremath{{\cal D}}}

\newcommand{\ci}{\ensuremath{{\cal I}}}

\newcommand{\cR}{\ensuremath{{\cal R}}}

\newcommand{\ti}{\tilde}

\newcommand{\al}{\alpha}

\newcommand{\ga}{\gamma}

\newcommand{\la}{\lambda}

\newcommand{\si}{\sigma}

\newcommand{\ep}{\varepsilon}

\newcommand{\R}{\ensuremath{{\mathbb R}}}
\newcommand{\N}{\ensuremath{{\mathbb N}}}
\newcommand{\B}{\ensuremath{{\mathbb B}}}

\newcommand{\C}{\ensuremath{{\mathbb C}}}

\newcommand{\cC}{{\mathcal C}}

\newcommand{\downto}{\downarrow}

\newcommand{\tensor}{\otimes}

\DeclareMathOperator{\VolBB}{Vol\mathbb{B}}

\DeclareMathOperator{\inj}{inj}

\newcommand{\beq}{\begin{equation}}
\newcommand{\eeq}{\end{equation}}
\newcommand{\beqa}{\begin{equation}\begin{aligned}}
\newcommand{\eeqa}{\end{aligned}\end{equation}}
\newcommand{\brmk}{\begin{rmk}}
\newcommand{\ermk}{\end{rmk}}
\newcommand{\partref}[1]{\hbox{(\csname @roman\endcsname{\ref{#1}})}}
\newcommand{\half}{\frac{1}{2}}

\newcommand{\Rm}{{\mathrm{Rm}}}
\newcommand{\Ric}{{\mathrm{Ric}}}

\newcommand{\Sc}{{\mathrm{R}}}

\usepackage{soul}

\newcommand{\K}{{\ensuremath{\mathrm{K_{IC_1}}}}}
\newcommand{\I}{{\ensuremath{\mathrm{IC_1}}}}

\newcommand{\twopartcond}[4]
{ \setstretch{1.5}
	\left\{
		\begin{array}{ll}
			#1 & \mbox{on } #2 \\ 
			#3 & \mbox{on } #4
		\end{array}
	\right.
}

\title{Pyramid Ricci Flow in Higher Dimensions}
\author{Andrew D. McLeod and Peter M. Topping}
\date{ \today}

\begin{document}

\usetagform{red}

\maketitle

\begin{abstract}
In this paper, we construct a pyramid Ricci flow starting with a complete Riemannian manifold $(M^n,g_0)$ that is PIC1, or more generally satisfies a lower curvature bound $\K\geq -\al_0$. 
That is, instead of constructing a flow on $M\times [0,T]$, we construct it on a subset of space-time that is a union of parabolic cylinders
$\B_{g_0}(x_0,k)\times [0,T_k]$ for each $k\in\N$, where $T_k\downto 0$,
and prove estimates on the curvature and Riemannian distance.
More generally, we construct a pyramid Ricci flow starting with any noncollapsed \I-limit space, and use it to establish that such limit spaces are globally homeomorphic to smooth manifolds via homeomorphisms that are locally bi-H\"older.
\end{abstract}


\section{Introduction}
\label{intro}



A central issue in differential geometry is to understand Riemannian manifolds with lower curvature bounds. One of the  important tasks in this direction is to understand the topological implications of such geometric bounds. Another, 
which is the main focus of this paper, 
is to understand the structure of Gromov-Hausdorff limits of sequences of manifolds satisfying a uniform lower curvature bound. 

There is some choice as to the precise notion of curvature  bound to consider. 
Imposing a uniform lower bound on the sectional curvatures gives limits that are Alexandrov spaces, studied since the middle of the twentieth century, and about which we now have a great deal of information, e.g. \cite{BGP92}. 
In practice, we often know  a uniform lower bound not for each sectional curvature, but for a suitable average of sectional curvatures, and the instance that has received the most attention is the case of limits of manifolds with a uniform lower Ricci bound. Such \emph{Ricci limit spaces} have been studied extensively since the work of Cheeger-Colding, starting in the 1990s, and have been widely applied, for example in the study of Einstein manifolds, \cite{Cheeger}. One result that is particularly relevant to the present paper is the topological regularity of (non-collapsed) three-dimensional Ricci limit spaces in the sense that they are globally homeomorphic to smooth manifolds via homeomorphisms that are locally bi-H\"older \cite{Topping2, MT18, Hochard}.

This paper is concerned principally with a way of averaging    
sectional curvatures that is less familiar than Ricci curvature.
Positivity of this average is generally referred to as PIC1,
with this concept first appearing in 
the seminal work of Micallef and Moore \cite{MM88} that was principally concerned with the weaker notion of \emph{positive isotropic curvature}, itself now abbreviated as PIC.
We will give the definition and basic properties of PIC1, and its nonnegative version WPIC1 (sometimes called NIC1) in Section \ref{PIC1sect}.

The PIC1 condition is natural for multiple reasons. 
To begin with, it can be naturally compared with other  curvature conditions. For example, it is implied by $\frac14$-pinching,
positive curvature operator, 2-positive curvature operator, and positive complex sectional curvature separately \cite[\S 5]{MM88}.
As we recall in Section \ref{PIC1sect},
the PIC1 condition implies that $\Ric>0$, so volume comparison and compactness are at our disposal, and in three dimensions the conditions are equivalent. 
Moreover, for many purposes PIC1 appears to be just the right condition to control the topology of the underlying manifold and
the regularity of limit spaces.
Related to this is that PIC1 interacts very well with the Ricci flow.
First, it is preserved under the flow \cite{N10, Ham97, BS09, Wilking}. Second, PIC1 is strong enough to guarantee that closed manifolds flow under (renormalised) Ricci flow to spherical space forms, as shown by Brendle \cite{B08}, generalising the earlier work of Hamilton \cite{Ham82}, B\"ohm-Wilking \cite{BW08} and Brendle-Schoen \cite{BS09}. In the noncompact case, it is tempting to make the following conjecture.

\begin{conj}
\label{PIC1conj}
If $(M,g_0)$ is a smooth, complete, $n$-dimensional Riemannian manifold, $n\geq 3$, satisfying the WPIC1 condition, then there exists a smooth 
WPIC1 Ricci flow $g(t)$ on $M$ for $t\in [0,T)$, some $T>0$, such that $g(0)=g_0$.
\end{conj}


\noindent
The three-dimensional version of this conjecture, where the assumption is of nonnegative Ricci curvature,  has been considered for a long time.
However, the conjecture should be false if we were to assume only nonnegative Ricci curvature in higher dimensions rather than WPIC1.
Under the much stronger condition of nonnegative complex sectional curvature, Cabezas-Rivas and Wilking managed to start the Ricci flow in \cite{CRW15}. In three-dimensions this would correspond to nonnegative sectional curvature, which is much more restrictive that nonnegative Ricci curvature despite the close links between Ricci and sectional curvatures in this dimension.
Under the additional asymptotic condition of maximal volume growth, the flow was started in \cite{HL18}.


For the remainder of the paper we will mainly consider a weaker condition than PIC1 in the sense that we ask that all the complex sectional curvatures corresponding to degenerate 2-planes (see Section \ref{PIC1sect}) are not necessarily positive, as for PIC1, but 
are bounded below by $-\al<0$, say. 
If we write $\K$ for the function acting on the space of all degenerate 2-planes in fibres of $T^\C M$, and returning the corresponding complex sectional survature (see Section \ref{PIC1sect}) then this condition is written $\K\geq -\al$.

In contrast to Conjecture \ref{PIC1conj}, it seems that a complete manifold 
can satisfy $\K\geq -\ep<0$ but not admit a Ricci flow even for a short time. Indeed it is easy to generalise the example in \cite{ICM2014} to higher dimensions.
Alternatively we can take a product $\R\times S^{n-1}$ with the 
warped product metric $dr^2+f(r)g_{S^{n-1}}$, where $f(r)$ is a suitable slowly-decreasing function with $f(r)\to 0$ as $r\to\infty$.
Because the manifold is collapsing at infinity, and looks like a product metric at the curvature scale, the Ricci flow would intuitively like to pinch the $S^{n-1}$ component at a given value of $r$ in a time of order $f(r)\to 0$.

One can avoid these difficulties by preventing the manifold from being singular at infinity, for example by imposing bounded curvature or a global noncollapsing condition (i.e. that the volume of any unit ball has a uniform positive lower bound). 
Alternatively, one can flow only locally.
See \cite{shi, Hochard, Topping2, YL18, BCRW19, HL18}.
In this paper, we take the approach of \emph{pyramid Ricci flows}, as introduced in \cite{MT18},
to flow on a pyramid shaped subspace of space-time with controlled geometry:

\begin{theorem}[\bf{\emph{Global pyramid Ricci flows}}]
\label{mollification}
Let $\al_0 , v_0 > 0,$ $n \in \N$ with $n \geq 3$.
Suppose that $(M,g_0)$ is an $n$-dimensional complete Riemannian manifold
with $\K[g_0]\geq -\al_0$ throughout, and $\VolBB_{g_0}(x_0,1)\geq v_0$ for some $x_0\in M$.
Then there exist increasing sequences $C_j \geq 1$ and $\al_j > 0$ and a decreasing sequence $T_j > 0,$ all defined for $j \in\N$, and depending only on $n$, $\al_0$ and $v_0$, for which the following is true.

There exists a smooth Ricci flow $g(t),$ defined on a subset of spacetime that contains, for each $j \in \N,$ the cylinder $\B_{g_0} ( x_0 , j) \times \left[0, T_j \right],$ 
satisfying that $g(0)=g_0$ throughout $M$, and further that, again for each $j \in \N,$
\beq\label{moll concs} \twopartcond
		{\K[g(t)] \geq -\al_j } 
		{\B_{g_0} ( x_0 , j) \times \left[ 0 , T_j \right]}
		{ \left| \Rm \right|_{g(t)} \leq \frac{C_j}{t}} 
		{\B_{g_0} ( x_0 , j ) \times \left( 0 , T_j \right].}
\eeq
\end{theorem}
\noindent
Thus, as in \cite{MT18}, the domain of definition of the Ricci flow starts with the whole manifold, but shrinks to avoid the singularities that we  envisage in the example above.

It was shown in \cite{MT18}, by proving an appropriate compactness result, that the existence of global pyramid Ricci flows such as those in  Theorem \ref{mollification} follows from the construction of
local pyramid Ricci flows as in the theorem below.
(Note as in Section \ref{PIC1sect} that $\I$ lower bounds imply lower Ricci bounds, so this implication follows as in 
\cite[Theorem 1.3]{MT18}.)



\begin{theorem}[\bf{\emph{Local pyramid Ricci flows}}]
\label{Ricci Flow} 
Let $\al_0 , v_0 > 0,$ $n \in \N$ with $n \geq 3$. 
Suppose $(M,g_0)$ is an $n$-dimensional complete Riemannian manifold
with $\K[g_0]\geq -\al_0$ throughout, and $\VolBB_{g_0}(x_0,1)\geq v_0$ for some $x_0\in M$.
%
Then there exist increasing sequences $C_k \geq 1$ and  $\al_k>0$, and a decreasing sequence $T_k> 0$, all defined for $k \in \N$, and  depending only on $n$, $\al_0$ and $v_0,$ such that the following is true. 
For any $ l \in \N$ there exists a smooth Ricci flow solution $g_l(t),$ defined on a subset $\cd_l$ of spacetime given by 
$$\cd_l := \bigcup_{k=1}^{l} \B_{g_0} (x_0 , k) \times \left[ 0 , T_k \right],$$
with $g_l(0) = g_0$ on $\mathbb{B}_{g_0} (x_0 , l)$, and satisfying,
for each $k\in \{1,\ldots, l\}$,
\beq\label{conc_1} \twopartcond
		{\K[g_l(t)] \geq -\al_k } 
		{\B_{g_0} ( x_0 , k) \times \left[ 0 , T_k \right]}
		{\left| \Rm \right|_{g_l(t)} \leq \frac{C_k}{t}} 
		{\B_{g_0} ( x_0 , k ) \times \left(0,T_k\right].}
\eeq
\end{theorem}
\noindent
As we increase $l$, the local pyramid Ricci flows $g_l(t)$ are defined on a larger and larger domain in space-time. If each flow extended the previous one, then we could take a union of them to obtain the flow required for Theorem \ref{mollification}. However, the flows are emphatically not unique, and instead we need to take a limit of a subsequence of the flows.

One novelty of pyramid Ricci flows, 
which is essential in order to be able to appeal to compactness and take a limit of a subsequence of the flows as $l\to\infty$, is that the shape of 
the domain $\cd_l$ intersected with $\B_{g_0} (x_0 , r)\times [0,\infty)$ is independent of $l\geq r$. We pay for this by ending up with curvature bounds \eqref{conc_1} that deteriorate as $k$ increases. In contrast, \emph{partial} Ricci flows, with instead uniform $C/t$ curvature bounds, defined on subsets of space-time,
were considered by Hochard \cite{Hochard}.

In fact, the shape of the domains $\cd_l$ depends on the initial data $g_0$ only in terms of $n$, $\al_0$ and $v_0$.
Just as in \cite{MT18} (cf. the proof of Theorem 5.1 there) 
this means that we can apply Theorem \ref{Ricci Flow} to pointed manifolds $({\cal M}_l,g_l,x_l)$
approximating a limit space, and appeal to compactness to obtain a Ricci flow starting at the given limit space.
Whereas in \cite{MT18} we worked with Ricci limit spaces, here we work with \emph{\I- limit spaces}. 

\begin{definition}
We call a complete metric space $(X,d)$ a noncollapsed 
\emph{\I-limit space} corresponding to  $\al_0>0$, $v_0>0$ and $n\in\N$ with $n\geq 3$,
if it arises within a pointed Gromov-Hausdorff limit 
$$(\m_i,g_i,x_i)\to (X,d,x_\infty)$$
of a sequence of pointed $n$-dimensional Riemannian manifolds such that
$\VolBB_{g_i}(x_i,1)\geq v_0>0$ and $\K[g_i]\geq -\al_0$.
\end{definition}

\begin{theorem}[{\bf{Pyramid Ricci flow from a \I-limit space}}]
\label{PIC1_pyramid_RF}
Suppose that $(X,d)$ is a \I-limit space 
corresponding to  $\al_0>0$, $v_0>0$ and $n\in \N$ with $n\geq 3$.
Then there exist increasing sequences $C_k \geq 1$ and $\al_k > 0$ and a decreasing sequence $T_k > 0,$ all defined for $k \in\N$, and depending only on $n$, $\al_0$ and $v_0,$ for which the following holds.

There exist a smooth $n$-manifold $M,$ a point $x_0 \in M,$ 
a complete distance metric $d : M \times M \rightarrow [0,\infty)$ generating the same topology as we already have on $M,$  
and a smooth Ricci flow $g(t)$ 
defined on a subset of spacetime $M \times (0,\infty)$ that contains $\B_{d} ( x_0 , k) \times \left(0, T_k \right]$ for each $k \in \N,$ 
with $d_{g(t)} \rightarrow d$ locally uniformly on $M$ as $t \downarrow 0,$ 
such that $(M,d)$ is isometric to $(X,d)$.
Moreover, for any $k \in \N,$ 
\begin{equation}
\label{r_d_concs} 
\twopartcond 
		{\K[g(t)] \geq -\al_k } 
		{\B_{d} ( x_0 , k ) \times \left(0,T_{k}\right]}
		{| \Rm |_{g(t)} \leq \frac{C_{k}}{t}} 
		{\B_{d} ( x_0 , k ) \times \left(0,T_{k}\right].}
\end{equation}
Finally, if $g$ is any smooth complete Riemannian metric on $M$ then the identity map 
$(M,d)\to (M,d_g)$ is locally bi-H\"older.
\end{theorem}

\noindent
Thus Ricci flow gives enough global regularisation, as in \cite{MT18}, to establish that \I-limit spaces are manifolds:

\begin{theorem}[{\bf \em \I-limit spaces are globally smooth manifolds}]
\label{initial_main_thm} 
Let $\al_0 , v_0 > 0$, and $n \in \N$ with $n \geq 3$.
Suppose that $\left( \m_i , g_i,x_i \right)$, for $i \in \N$, is a sequence of $n$-dimensional pointed Riemannian manifolds with
$\VolBB_{g_i}(x_i,1)\geq v_0>0$ and $\K[g_i]\geq -\al_0$.
Then there exist a smooth $n$-manifold $M$, a point $x_0\in M$,
and a complete distance metric $d: M \times M \to [0,\infty)$ generating the same topology as $M$ such that after passing to a subsequence in $i$ we have 
$$\left( \m_i , d_{g_i} ,x_i \right)\to\left( M , d , x_0 \right),$$
in the pointed Gromov-Hausdorff sense, and if $g$ is any smooth complete Riemannian metric on $M$ then the identity map 
$(M,d)\to (M,d_g)$ is locally bi-H\"older.
\end{theorem}

\noindent
The $n=3$ case of this result was proved in \cite{MT18}, extending the work in \cite{Topping2, Hochard} that obtained a local bi-H\"older 
description of noncollapsed Ricci limit spaces as smooth manifolds. 
The proof of the local description extends verbatim to higher dimensions
once the lower Ricci curvature bounds of \cite{Topping1} have been suitably generalised, and this was done by Y. Lai based on extensions of the 
curvature estimates of Bamler, Cabezas-Rivas and Wilking \cite{YL18, BCRW19}.

\begin{remark}
\label{other_cones}
For the remainder of the paper, we use a   slightly different way of writing the curvature condition $\K[g]\geq -\al$ that is more consistent with the literature on which we draw. We will write $\cR_g$ for the curvature operator, and $\cC_{\I}$ for the closed cone of algebraic curvature operators satisfying the WPIC1 condition; see e.g.
\cite{Wilking} for details.
Then the condition $\K[g]\geq -\al$ can be written 
$\cR_g +\al \ci \in \cC_{\I}$.

Similarly, we can define the cones $\cC_{CO}$ and $\cC_{CSC}$ corresponding to the curvature conditions positive curvature operator and positive complex sectional curvature.
It is not hard to check that $\cC_{CO}\subset\cC_{\I}$
and $\cC_{CSC}\subset \cC_{\I}$.
Consequently, Theorems \ref{mollification} and \ref{Ricci Flow} can both be applied when the assumed curvature condition 
$\cR_{g_0} +\al_0 \ci \in \cC_{\I}$ is strengthened to 
$\cR_{g_0} + \al_0 \ci \in \cC$ for any cone
$\cC \in \left\{ \cC_{CO} , \cC_{CSC} \right\}.$
At first glance, the resulting flows only have a $\I$-lower bound,
as in \eqref{moll concs} and \eqref{conc_1},
however
we can use Lemmas \ref{loc_lemma_analogue} and \ref{DB} to improve 
these to $CO$ or $CSC$ lower bounds respectively,
after adjusting the sequences $C_j$, $\al_j$ and $T_j$ 
(cf. the proofs of Theorems 1.3 and 5.1 in \cite{MT18}, for example).


An examination of Hochard's Proposition II.2.6 in \cite{Hoc19} reveals it is true for
the cone $\cC_{2CO}$ of two-positive curvature operators.
Thus Lemma \ref{DB} is valid for this cone, and since
$\cC_{2CO} \subset \cC_{\I}$, the above strategy would also allow us to apply 
Theorems \ref{mollification} and \ref{Ricci Flow} under the assumed curvature condition
$\cR_{g_0} + \al_0 \ci \in \cC_{2CO}$
with correspondingly stronger conclusions.
\end{remark}

\vskip5pt

\noindent
The remainder of the paper is devoted to the proof of Theorem \ref{Ricci Flow} from which the other results follow as discussed  above.
A key ingredient in the proof is Hochard's local version of the estimates of Bamler, Cabezas-Rivas and Wilking \cite{Hoc19, BCRW19} that generalises the Ricci lower bounds of the \emph{double bootstrap} lemma from \cite{Topping1}, see Proposition \ref{Hoc_II.2.6} and Lemma \ref{DB} below.
The proof will be completed in Section \ref{constants} by iterating 
a new \emph{Pyramid extension lemma \ref{PEL}}.

\vskip10pt

\noindent
\emph{Acknowledgements:} 
This work was supported by EPSRC grant number EP/K00865X/1 
and an EPSRC Doctoral Prize fellowship number EP/R513143/1.
The first author would like to thank Felix Schulze for helpful discussions on this topic.
The second author would like to thank Mario Micallef and Andrea Mondino for useful conversations about PIC.

\section{A brief review of PIC1}
\label{PIC1sect}

As mentioned in Section \ref{intro}, PIC1 and its related curvature conditions correspond to the positivity of certain averages of sectional curvatures of  a given \emph{Riemannian} manifold $(M,g)$, just as for Ricci curvature.
To express which averages to take in the most natural way, we 
complexify the tangent bundle, i.e. consider $T^\C M:=TM\tensor_\R \C$, which essentially consists of elements $X+iY$ for vectors $X,Y\in T_pM$.
Just as the usual sectional curvature assigns a real number 
$\Rm(X,Y,X,Y)$ to each two-dimensional linear plane $\si\subset T_pM$ 
spanned by an orthonormal pair $X,Y$, 
the \emph{complex sectional curvature} corresponding to a two-dimensional  complex linear subspace of $T_p^\C M$ spanned by $v,w$ with 
$\langle v,v\rangle = \langle w,w\rangle = 1$ and $\langle v,w\rangle = 0$ is  $\Rm(v,w,\overline v,\overline w)\in \R$.
Here, $\langle v,w\rangle:= (v,\overline w)$ is the usual Hermitian inner product corresponding to the complex linear extension $(\cdot,\cdot)$ of the Riemannian metric $g$, and we have implicitly extended the curvature tensor by complex linearity.

Asking that a manifold has 
nonnegative
complex sectional curvature, i.e. that 
the number computed above is nonnegative for each 
two-dimensional  complex linear subspace $\si$ of fibres of $T^\C M$,
is a strong condition that coincides with a condition introduced by
Brendle-Schoen \cite{BS09, NW07} that is often called WPIC2. It is clearly more restrictive than nonnegative sectional curvature, since we are always free to pick $\si$ consisting only of real elements. 
Rephrased, we arrive at the more general condition of nonnegative sectional curvature by asking for nonnegativity of the complex sectional curvatures corresponding only to complex linear two-planes
$\si$ for which $\si=\overline\si$, where
$\overline\si$ is the linear two-plane obtained by taking the complex conjugate of each element of $\si$.

In practice, we would like to restrict to different subsets of 
all complex linear two-planes $\si$ by comparing $\si$ to $\overline\si$ in other ways.
If $\si$ and $\overline\si$ are orthogonal in the sense that every element of $\si$ is orthogonal to every element of $\overline\si$ with respect to the Hermitian inner product, then we say that $\si$ is totally isotropic. A single vector $v\in T^\C M$ is said to be isotropic
if $(v,v)=0$, and this is easily seen to be equivalent to being a complex multiple of some $e_1+ie_2$ with $e_1,e_2$ orthonormal.
It can be shown that an equivalent formulation of $\si$ being totally isotropic is that $\si$ is spanned by elements $e_1+ie_2$ and $e_3+ie_4$ for some orthonormal collection $e_1,e_2,e_3,e_4\in T_pM$.
A third formulation would be that every $v\in\si$ is isotropic \cite{MM88}.
Positivity of all the complex sectional curvatures corresponding to such $\si$ is the condition of \emph{positive isotropic curvature} (PIC) mentioned in the introduction. 
Nonnegativity (i.e. weak positivity) of all such curvatures is called WPIC (or NIC).
One needs to be working in dimension at least $4$ for this to make sense.

In practice, the PIC condition is too weak for many purposes.
We can strengthen the condition by considering all the planes $\si$ whose projection onto $\overline\si$ may  be the zero element (as for PIC) 
or more generally may be of complex dimension one (i.e. it is not of dimension two). Equivalently, we consider all \emph{degenerate} $\si$, i.e. that contain an element $v$ such that $(v,w)=0$ for all $w\in\si$.
Positivity of all  curvatures corresponding to such $\si$ is the condition known as PIC1. 
Nonnegativity of all such curvatures is known as WPIC1 or NIC1.
The terminology arises because an equivalent way of stating the PIC1 condition is to say that $(M,g)\times \R$ 
satisfies the PIC condition.


By picking an arbitrary orthonormal collection 
$\{e_1,e_2,e_3\}$ and considering the plane spanned by $e_1$ and the isotropic vector $v:=e_2+ie_3$, we see that the PIC1 condition implies that 
$\Rm(e_1,e_2,e_1,e_2)+\Rm(e_1,e_3,e_1,e_3)>0$. In particular, it implies $\Ric>0$, so volume comparison and compactness can be applied. 
On the other hand, in three dimensions, any of the degenerate `PIC1 planes' can be viewed as the span of $e_1$ and $e_2+ie_3$ for some
orthonormal collection $\{e_1,e_2,e_3\}$, and so PIC1 is equivalent to positive Ricci curvature in this dimension.


\section{Local Flows and Curvature Estimates}
\label{variants}
We begin by recording some minor variants of known local estimates for Ricci flow, 
and a known local existence result.
We first examine the consequences of a flow $g(t)$ satisfying  
$\cR_{g(t)} + \gamma \mathcal{I} \in \cC_{\I}$ throughout a local region in space-time.
The $n=3$ case of following result can be found in \cite[Lemma 4.1]{Topping2}, or in this form in \cite[Lemma A.1]{MT18}. 
By developing the curvature estimates of \cite{BCRW19}, and using an extension of the work of Perelman \cite[\S 11.4]{P02}, also from \cite{BCRW19}, the same proof extends to higher dimensions, as shown by Y. Lai. The statement we give differs from \cite[Lemma 3.4]{YL18} mainly in that $C_0$ does not depend on $\ga$, and will follow easily from Lai's statement by scaling.


\begin{lemma} 
\label{loc_lemma_analogue}
Given any $n \in \N$ and $v_0 > 0$, there exists a constant 
$C_0 = C_0 (n, v_0) \geq 1$ such that the following is true.
Let $\left( M , g(t) \right)$ be a smooth $n$-dimensional Ricci flow, defined for all times $t \in [0,T],$ such that for some $p \in M$ and $\ep>0$ we have $\B_{g(t)} (p,\ep) \subset \subset M$ for each $t \in [0,T],$
and so that for any $r \in (0,\ep]$ we have that 
$\VolBB_{g(0)} (p,r) \geq v_0 r^n $. 
Further assume that for some $\gamma > 0$ and all $t \in [0,T]$
we have 
\beq
	\label{alt_curv_assump}
		\cR_{g(t)} + \gamma \mathcal{I} \in \cC_{\I} \qquad \text{on} 
		\qquad \bigcup_{s \in [0,T]} \B_{g(s)} (p,\ep).
\eeq 
Then there exists $S = S(n, v_0, \gamma, \ep) > 0$ such that for all $0 < t \leq \min \left\{ S , T \right\}$ we have both
\beq
	\label{rescale_1_conc}
		|\Rm|_{g(t)}(p) \leq \frac{C_0}{t} \qquad \text{and} 
		\qquad \inj_{g(t)}(p) \geq \sqrt{ \frac{t}{C_0} }.
\eeq
\end{lemma}

\begin{proof}[Proof of Lemma \ref{loc_lemma_analogue}]
Without loss of generality we may  assume that $\ep \leq 1$ (otherwise replace $\ep$ by $1$)
and that $\gamma \geq \frac{1}{\ep^2}$ (otherwise replace $\ga$ by $\frac{1}{\ep^2}$).
Consider the rescaled flow $g_{\gamma} (t) := \gamma g ( \frac{t}{\gamma})$ for $0 \leq t \leq \gamma T.$ 
Since $\gamma^{-\frac{1}{2}} \leq \ep,$ we have 
\beq
	\label{volume_req}
		\VolBB_{g_{\gamma}(0)}(p,1) = \gamma^{\frac{n}{2}} \VolBB_{g(0)} ( p , \gamma^{-\frac{1}{2}} ) \geq \gamma^{\frac{n}{2}} \gamma^{-\frac{n}{2}} v_0 
		=v_0.
\eeq
Moreover, for any $0 \leq t \leq \gamma T$, again using that 
$\gamma^{-\frac{1}{2}} \leq \ep,$ 
\beq
	\label{compact_req}
		\B_{g_{\gamma}(t)} ( p , 1) 
		= \B_{g \left( \frac{t}{\gamma} \right)} \left( p , \gamma^{-\frac{1}{2}} \right)
		\subset \B_{g \left( \frac{t}{\gamma} \right)} (p,\ep) \subset \subset M,
\eeq
which in turn tells us that
\beq
	\label{union_good}
		\bigcup_{s \in \left[ 0 , \gamma T \right]} \B_{g_{\gamma} (s)} (p,1)
		\subset 
		\bigcup_{s \in \left[ 0 , T \right]} \B_{g (s)} (p,\ep).
\eeq
Together, \eqref{alt_curv_assump} and \eqref{union_good} yield that the rescaled flow $g_{\gamma}(t)$ satisfies
\beq
	\label{curv_req}
		\cR_{g_{\gamma}(t)} + \mathcal{I} \in \cC_{\I} 
		\qquad \text{on} \qquad 
		\bigcup_{s \in \left[ 0 , \gamma T \right]} \B_{g_{\gamma} (s)} (p,1) 
		\qquad \text{for all} \qquad t \in \left[ 0 , \gamma T\right].
\eeq
Combining \eqref{volume_req}, \eqref{compact_req} and \eqref{curv_req} we have the hypotheses to be able to apply 
Lemma 3.4 in \cite{YL18}. 
Doing so gives us constants $C_0 = C_0 (n, v_0) \geq 1$ and $S_0 = S_0 (n , v_0) > 0$ such that for all $0 < t \leq \min \left\{ \gamma T , S_0 \right\}$ 
the conclusion \eqref{rescale_1_conc} hold for $g_\ga(t)$ instead of $g(t)$. But these estimates are invariant under parabolic scaling, so the lemma holds with $S=S_0/\ga$.
\end{proof}
\vskip 4pt
\noindent
Next, we record a result that generalises the double-bootstrap lemma of Simon and the second author, see Lemma 9.1 in \cite{Topping1} and Lemma 4.2 in \cite{Topping2}, to higher dimensions.
This result is a minor adaptation of Proposition II.2.6 in the thesis of R. Hochard \cite{Hoc19} (see Proposition \ref{Hoc_II.2.6} here).

\begin{lemma}[{Propagation of lower curvature bounds; Variant of Proposition II.2.6 in \cite{Hoc19}}]
\label{DB}
Let $n \in \N$ with $n \geq 3$ and $c_0 , \al_0 > 0.$ Suppose that $( M , g(t) )$ is a smooth $n$-dimensional Ricci flow, 
defined for $0 \leq t \leq T,$ and satisfying that for some point $x \in M$ and $\ep>0$ we have $\B_{g(0)} (x , \ep) \subset \subset M.$ 
We further assume that
\beq
	\label{DB_hyp_1}
		| \Rm |_{g(t)} \leq \frac{c_0}{t} \qquad \text{and} \qquad \inj_{g(t)} \geq \sqrt{ \frac{t}{c_0} } 
\eeq
throughout $\B_{g(0)} (x , \ep) \times (0, T]$ and that
\beq
	\label{DB_hyp_2}
		\cR_{g(0)} + \al_0 \cal{I} \in \cal{C}
\eeq 
throughout $\B_{g(0)} (x,\ep),$
where $\cC$ is one of the invariant curvature cones 
$\cC_{CO}, \cC_{\I}$ or $\cC_{CSC}$,
that are described in Remark \ref{other_cones}.
Then there exist constants $S = S(n,c_0,\al_0,\ep) >0$ and 
$K = K ( n , c_0 , \al_0, \ep ) > 0$ such that
\beq
	\label{DB_conc}
		\cR_{g(t)} (x) + K \cal{I} \in C
\eeq
for all times $0 \leq t \leq \min \left\{ S , T \right\}.$
\end{lemma}


\begin{proof}[Proof of Lemma \ref{DB}]
By making a single parabolic rescaling, it suffices to prove the lemma in the case
that $\ep=4$.

Regardless of which curvature cone $\cC$ we are working with, we always have the inclusion $\cC \subset \cC_{\I}$. 
Thus there exists $\lambda = \lambda(n ,\al_0) \geq 1$ such that 
$\Ric_{g(0)} \geq -\lambda$ and $\Sc_{g(0)} \geq -\lambda$ 
throughout $\B_{g(0}(x,4)$. 

By the shrinking balls lemma \ref{nested balls}, for sufficiently small $S\in (0,1]$, depending only on $n$ and $c_0$, we can be sure that for all $z\in \B_{g(0)}(x,2)$ we have
$\B_{g(t)}(z,1)\subset \subset \B_{g(0)}(z,2)\subset \B_{g(0)}(x,4)$
for all $0\leq t\leq \min\{T,S\}$. 
We will reduce $S>0$ further below, with the understanding that it can only depend on
$n$, $c_0$ and $\al_0$.
By Lemma 8.1 of \cite{Topping1}, applied with $x_0$ there equal to $z$ here, we can deduce  (for possibly smaller $S$) that
$\Sc_{g(t)}\geq -2\la$ on $\B_{g(0)}(x,2)$ for all $0\leq t\leq \min\{T,S\}$.

This allows us to apply Proposition \ref{Hoc_II.2.6} to an appropriately parabolically scaled up version of $g(t)$, to deduce that 
$\cR_{g(t)} (x) + K \cal{I} \in C$ for some $K>0$ depending only on $n$, $c_0$ and $\la$, i.e. on $n$, $c_0$ and $\al_0$, for all $0\leq t\leq \min\{T,S\}$ (for possibly smaller $S$).
\end{proof}

\vskip 4pt
\noindent
We conclude this section by recording that it is possible to find a local solution to the Ricci flow, assuming a lower $\K$ bound.
This is the content of Theorem 1.1 in \cite{YL18}; we state a minor variant that is more convenient for our purposes.
In particular, we reduce the initial noncollapsedness hypothesis to a 
lower volume bound for a single unit ball, 
rescale the result to apply to any ball of radius 
strictly larger than one,
and add a lower injectivity radius bound 
to the conclusion.
The injectivity radii bounds are implicitly obtained within the proof of Theorem 1.1 in \cite{YL18}, 
and the following result simply makes these explicit.

\begin{theorem}[Local Existence; Variant of Theorem 1.1 in \cite{YL18}]
\label{loc_exist_YL}
Given $n \in \N$, $R\geq 1$  and $\ep, \al_0 , v_0 > 0$ there exist positive constants 
$C , \tau >0$,
both depending only on $n, \al_0, v_0, \ep$ and $R$, 
for which the following is true.
Let $\left( M , g_0, x_0 \right)$ be a smooth pointed Riemannian $n$-manifold, and 
suppose that $\B_{g_0} (x_0 , R+\ep) \subset \subset M$ and 
\beq
	\label{loc_exist_curv_assump}
		\cR_{g_0} + \al_0 \mathcal{I} \in \cC_{\I} 
		\qquad \text{throughout} \qquad
		\B_{g_0} (x_0 , R+\ep)
\eeq
and 
\beq
	\label{loc_exist_vol_assump}
		\VolBB_{g_0} (x_0,1) \geq v_0.
\eeq
Then there exists a smooth Ricci flow $g(t)$ defined for $0 \leq t \leq \tau$ on $\B_{g_0} (x_0 , R),$ with $g(0) = g_0$ where defined, such that for all $0 < t \leq \tau$ we have
\beq
	\label{loc_exist_conc}
		|\Rm|_{g(t)} \leq \frac{C}{t} 
		\qquad \text{and} \qquad
		\inj_{g(t)} \geq \sqrt{\frac{t}{C}}
		\qquad \text{and} \qquad 
		\cR_{g(t)} + C  \mathcal{I} \in \cC_{\I} 
\eeq
throughout $\B_{g_0} (x_0 , R).$
\end{theorem}

\begin{proof}[Proof of Theorem \ref{loc_exist_YL}]
By parabolically scaling up the flow by a factor depending only on $\ep$ and $\al_0$, we may assume that 
$\al_0\leq 1$ and $\ep\geq 6$.
Note that as long as we scale up, Bishop-Gromov will ensure that the volume condition \eqref{loc_exist_vol_assump} will be satisfied for some new $v_0$ depending on the old $v_0$, $\al_0$ and $n$.
(Recall that the curvature condition \eqref{loc_exist_curv_assump} implies a lower Ricci bound.)

In fact, repeatedly applying Bishop-Gromov tells us that, for the scaled up flow, for all $x\in \B_{g_0}(x_0,R+4)$, 
we have a positive lower bound for $\VolBB_{g_0}(x,1)$ that depends only on $v_0$, $\al_0$, $n$ and $R$. 
More generally we obtain such a lower bound for $r^{-n}\VolBB_{g_0}(x,r)$ for any $r\in (0,1]$, with the bound independent of $r$.

This puts us in a position to apply \cite[Theorem 1.1]{YL18} with $s_0=R+5$ to obtain a Ricci flow on $\B_{g_0} ( x_0 , R + 3 )$ and
deduce all the conclusions aside from the injectivity radius bound on \eqref{loc_exist_conc}.
However, this follows easily from Lemma \ref{loc_lemma_analogue} as follows.
The shrinking balls lemma \ref{nested balls} 
allows us to deduce that if
$x \in \B_{g_0} ( x_0 , R )$ then 
$\B_{g(t)} (x,1) 
\subset \subset 
\B_{g_0} ( x , 2 )
\subset 
\B_{g_0} ( x_0 , R + 2 )$,
for all times $0 \leq t \leq \tau$, if we reduce $\tau>0$  appropriately.
Thus 
$\cR_{g(t)} + C \mathcal{I} \in \cC_{\I}$ throughout 
$\bigcup_{s \in [0 , \tau]} \B_{g(s)} \left(x,1\right)$
for all $t \in [ 0 , \tau ]$, and we are directly in a position to apply 
Lemma \ref{loc_lemma_analogue} with $\ep$ there equal to $1$.
\end{proof}

\section{The Pyramid Extension Lemma}
\label{constants}

The following result is an analogue of the Pyramid Extension Lemma \cite[Lemma 2.1]{MT18} in  higher dimensions.
It can be considered an extension of the local existence theorem \ref{loc_exist_YL} of Y. Lai.

\begin{lemma}[\bf \em Pyramid Extension Lemma]
\label{PEL}
Let $\al_0 , v_0 > 0$ and $n \in \N$ with $n \geq 3$.
Suppose 
$(M,g_0,x_0)$ is a pointed complete Riemannian $n$-manifold 
such that 
$\VolBB_{g_0} (x_0 , 1) \geq v_0$ and  
$\cR_{g_0} + \al_0 \mathcal{I} \in \cC_{\I}$ throughout $M.$
Then there exist  increasing sequences
$C_k\geq 1$ and $\al_k>0,$ 
and a decreasing sequence 
$T_k> 0$, all defined for $k\in\N$ and depending only on $n, \al_0$ and $v_0$, with the following properties. 
\begin{compactenum}[1)]
\item
For each $k\in \N$ there exists a Ricci flow $g(t)$ on 
$\B_{g_0}(x_0,k)$ 
for $t\in [0,T_k]$ such that $g(0)=g_0$ where defined and so that 
$|\Rm|_{g(t)}\leq C_k/t$ and $\inj_{g(t)} \geq \sqrt{ t/C_k}$
for all $t\in (0,T_k]$ and 
$\cR_{g(t)} + \al_k \mathcal{I} \in \cC_{\I}$ for all $t\in [0,T_k]$. 
\item
Given any Ricci flow $\tilde g(t)$ on 
$\B_{g_0}(x_0,k+1)$ over a time interval $t\in [0,S]$, $S>0$, with $\tilde g(0)=g_0$ where defined, 
satisfying for all $t \in (0,S]$ that $|\Rm|_{\tilde g(t)}\leq C_{k+1}/t$ 
and $\inj_{\tilde g(t)} \geq \sqrt{ t/C_{k+1}}$, 
we may choose the Ricci flow $g(t)$ above to agree with the restriction of $\tilde g(t)$ to $\B_{g_0}(x_0,k)$ 
for times $t\in [0,\min\{S, T_k\}]$. 
\end{compactenum}
\end{lemma}

\begin{proof}[Proof of Lemma \ref{PEL}]
We will refine the strategy of Lemma 2.1 in \cite{MT18}, with the roles of the double bootstrap lemma 9.1 in \cite{Topping1} 
and the local lemma A.1 in \cite{MT18} being played by the propagation lemma \ref{DB} and  Lemma \ref{loc_lemma_analogue} here, respectively.

The first part of the lemma, giving the initial existence statement for $g(t)$, follows immediately by the local existence theorem \ref{loc_exist_YL} with $R=k$ and $\ep=1$, giving 
$C_k\geq 1$, $\al_k>0$ and $T_k>0$ depending only on $n$, $\al_0$, $v_0$ and $k$.
We will need to increase 
$C_k$ and $\al_k$, and decrease $T_k$, in order to establish the remaining claims of the lemma. 

Recall that $\cR_{g_0} + \al_0 \mathcal{I} \in \cC_{\I}$ throughout $M$ implies that $\Ric_{g_0} \geq -D$ throughout $M$ for some $D = D (n, \al_0) >0.$ Thus,
by Bishop-Gromov, for all $k\in\N$, there exists $v_k>0$ depending only on $k$, $n,$ $\al_0$ and $v_0$ such that 
if $x\in \B_{g_0}(x_0,k+1)$ and $r\in (0,1]$ then 
$\VolBB_{g_0}(x,r)\geq v_k r^n$.

We increase each $C_k$ to be at least as large as the constant $C_0$ retrieved from Lemma \ref{loc_lemma_analogue}  
with $v_0$ there equal to $v_k$ here. Note that we are not actually applying Lemma \ref{loc_lemma_analogue}, but simply retrieving a constant in preparation for its application at the end of the proof.
By inductively replacing $C_k$ by $\max\{C_k,C_{k-1}\}$ for $k=2,3,\ldots$, we can additionally assume that $C_k$ is increasing in $k$. 
Thus $C_k$ still depends 
only on $k$, $n$, $\al_0$ and $v_0$, and can be fixed for the remainder of the proof.

%
%
%

Suppose now that we would like to extend a Ricci flow $\tilde g(t)$.
Appealing to the propagation lemma \ref{DB} 
centred at each $x\in \B_{g_0}(x_0,k+\half)$, and with $\ep=\half$
and $c_0=C_{k+1}$, 
after possibly reducing $T_k>0$ and increasing $\al_k$, 
depending only on $n$, $C_{k+1}$ 
and $\al_0$,  and hence only on $n$, $k$, $\al_0$ and $v_0$ as before, 
we may assume that for all $t\in [0,\min\{S,T_k\}]$ 
we have 
$\cR_{\tilde g(t)} + \al_k \mathcal{I} \in \cC_{\I}$
throughout 
$\B_{g_0}(x_0,k+\half)$.

A first consequence of this estimate is that 
$\Ric_{\tilde{g}(t)} \geq - D_k $ 
over the same region of space-time,
for some $D_k > 0$ depending only on $n$ and $\al_k,$
i.e. only on $k$, $n$, $\al_0$ and $v_0.$
In turn, these Ricci lower bounds give  
better volume bounds via Lemma \ref{Volume 2}. 
We apply that result with $R=k+\frac13$ and $\ep=\frac16$ 
to obtain that for every 
$t\in [0,\min\{S,T_k\}]$, where we have reduced $T_k>0$ 
again without adding any additional dependencies, 
we have 
\beq
\label{new_inclusion}
\textstyle
\B_{\ti g(t)}(x_0,k+\frac13)\subset \B_{g_0}(x_0,k+\half),
\eeq
and 
$\VolBB_{\ti g(t)}(x_0,1)\geq \mu_k>0$, 
where 
$\mu_k$ depends only on $n$, $v_0$, $k$, and $\al_0$.

A further reduction of $T_k>0$ will ensure appropriate nesting of balls defined at different times. 
By the expanding balls lemma \ref{expanding balls}, 
exploiting again our lower Ricci bounds, we deduce that
\beq
\label{ball_inclusion_EBL}
\left\{
\begin{aligned}
&
\textstyle
\B_{g_0}(x_0,k+\frac15)\subset \B_{\ti g(t)}(x_0,k+\frac14)\\
&
\textstyle
\B_{g_0}(x_0,k)\subset \B_{\ti g(t)}(x_0,k+\frac{1}{20})
\end{aligned}
\right.
\eeq
and by the shrinking balls lemma \ref{nested balls}, we deduce that 
\beq
\label{new_SBL_inc}
\textstyle
\B_{\ti g(t)}(x,\frac{1}{6})\subset \B_{g_0}(x,\frac15)
\qquad
\text{for every }x\in \B_{g_0}(x_0,k),
\eeq
all for $t\in [0,\min\{S,T_k\}]$,
where $T_k>0$ has been reduced appropriately, without additional dependencies.

At this point we can temporarily fix $T_k$ and try to find our desired extension $g(t)$ of $\ti g(t)$ by considering 
$\ti g(\tau)$ for $\tau:=\min\{S,T_k\}>0$ and restarting the flow from there using the local existence theorem \ref{loc_exist_YL}.
(Note that $\tau$ is now fixed, but we will make  further reductions of $T_k$ later.)

In order to do so, note that $\ti g(\tau)$ satisfies the estimates 
$\cR_{\tilde g(\tau)} + \al_k \mathcal{I} \in \cC_{\I}$ 
on $\B_{g_0}(x_0,k+\frac12)\supset\B_{\ti g(\tau)}(x_0,k+\frac13)$, 
by \eqref{new_inclusion},
and $\VolBB_{\ti g(\tau)}(x_0,1)\geq \mu_k>0$.


The output of the local existence theorem \ref{loc_exist_YL}, 
applied with 
$M=\B_{g_0}(x_0,k+1)$, 
$R=k+\frac14$,
$\ep=\frac1{12}$,
$\al_0=\al_k$, and 
$g_0=\ti g(\tau)$, 
is that after reducing $T_k>0$, 
still depending only on $n$, $\al_0$, $k$ and $v_0$, there exists a Ricci flow $h(t)$ on $\B_{\ti g(\tau)}(x_0,k+\frac14)$ 
for $t\in [0,T_k]$, with $h(0)=\ti g(\tau)$ where defined, 
and such that 
$\cR_{h(t)} + \al_k \mathcal{I} \in \cC_{\I}$
(after possibly increasing $\al_k$ further, still depending only on 
$n$, $\al_0$, $k$ and $v_0$)
and $|\Rm|_{h(t)}\leq c_k/t$, where $c_k$ also depends only on 
$n$, $\al_0$, $k$ and $v_0$.
By the first inclusion of \eqref{ball_inclusion_EBL}, this flow is defined throughout 
$\B_{g_0}(x_0,k+\frac15)$. 

Define a concatenated Ricci flow on 
$\B_{\ti g(\tau)}(x_0,k+\frac14)\supset\B_{g_0}(x_0,k+\frac15)$ 
for $t\in [0,\tau+T_k]$ by
\begin{equation}
\label{def of g} 
g(t) := \left\{
\begin{aligned}
& { \tilde{g}(t) }\qquad & & {0 \leq t \leq \tau } \\
& {h\left(t- \tau \right)}\qquad & & {\tau < t \leq \tau+T_k }.
\end{aligned}
\right.
\end{equation} 
This already satisfies the required lower curvature estimate
$\cR_{g(t)} + \al_k \mathcal{I} \in \cC_{\I}$.

We claim that after possibly reducing $T_k>0$, without further dependencies, we have that for all $x\in \B_{g_0}(x_0,k)$, there holds the inclusion $\B_{g(t)}(x,\frac16)\subset\subset \B_{\ti g(\tau)}(x_0,k+\frac14)$, 
where the flow is defined, for all $t\in [0,\tau+T_k]$.

Because our curvature estimates currently deteriorate at time $\tau$, 
i.e. we do not yet have $c/t$ decay for all times,
we prove this claim separately for the cases $t \in [0,\tau]$ and $t \in (\tau , \tau+T_k]$.


For $t\in [0,\tau]$, the inclusion \eqref{new_SBL_inc} and the first inclusion of \eqref{ball_inclusion_EBL} tell us that 
(for a reduced $T_k>0$)
$$\textstyle
\B_{g(t)}(x,\frac16)\subset \B_{g_0}(x,\frac15)
\subset\subset \B_{g_0}(x_0,k+\frac15)\subset 
\B_{\ti g(\tau)}(x_0,k+\frac14),$$
so the claim holds up until time $\tau$.

Thus to prove the claim it remains to show that 
for all $x\in \B_{g_0}(x_0,k)$, there holds the inclusion 
$\B_{h(t)}(x,\frac16)\subset\subset \B_{h(0)}(x_0,k+\frac14)$ for all $t\in [0,T_k]$,
and by the second inclusion of \eqref{ball_inclusion_EBL}, it suffices to prove this 
for each $x\in \B_{h(0)}(x_0,k+\frac{1}{20})$.
But by the shrinking balls lemma \ref{nested balls}, after reducing $T_k>0$ 
we can deduce that 
$\B_{h(t)}(x,\frac16)\subset\subset\B_{h(0)}(x,\frac15) 
\subset\B_{h(0)}(x_0,k+\frac14)$
as required, thus proving the claim.


At this point we truncate the flow $g(t)$ to live only on the time interval $[0,T_k]$ (i.e. we chop off an interval of length $\tau$ from the end, not the beginning). 

The main final step is to apply Lemma \ref{loc_lemma_analogue}
to $g(t)$ with $M$ there equal to $\B_{\ti g(\tau)}(x_0,k+\frac14)$ here.
Using the claim we just proved, for every $x\in \B_{g_0}(x_0,k)$, 
after a possible further reduction of $T_k>0$, and with $C_k$ as fixed earlier, Lemma \ref{loc_lemma_analogue}, applied with $\ep=\frac16$, tells us that 
$|\Rm|_{g(t)}(x)\leq C_k/t$ and
$\inj_{g(t)} (x) \geq \sqrt{t/C_k}$
for all $t\in (0,T_k]$.
We finally have a sequence $T_k$ that does what the lemma asks of it, except for being decreasing. The monotonicity of $T_k$ and $\al_k$ can be arranged by
iteratively replacing $T_k$ by $\min\{T_k,T_{k-1}\}$,
and $\al_k$ by $\max\{\al_k,\al_{k-1}\}$, for $k=2,3,\ldots$.

\end{proof}

\noindent
The pyramid Ricci flows of Theorem \ref{Ricci Flow} are an immediate consequence of the Pyramid Extension Lemma \ref{PEL}:

\begin{proof}[Proof of Theorem \ref{Ricci Flow}]
By appealing to the Pyramid Extension Lemma \ref{PEL} we may 
retrieve increasing sequences $C_k \geq 1 , \al_k > 0$ and 
a decreasing sequence $T_k > 0$, 
all defined for $k \in \N$, 
and depending only on the given $n, \al_0$ and $v_0$.

To verify that these sequences meet the requirements of the theorem
we fix $l \in \N$ and use Lemma \ref{PEL} $l$ times to construct 
$g_l(t)$ as follows. 
First we use the first part of that lemma with $k=l$ to obtain 
an initial flow living on $\B_{g_0} (x_0 , l)$ for times
$t \in [0,T_l].$

Since $T_l \leq T_{l-1},$ we may appeal to the second part of Lemma
\ref{PEL} with $k=l-1$ to extend this flow to the longer time interval
$[0,T_{l-1}]$, albeit on the smaller ball $\B_{g_0} (x_0 , l-1).$

We repeat this process inductively for the remaining values of $k$ down until it is finally repeated for $k=1.$ 
The resulting smooth Ricci flow $g_l(t)$ is now defined, for each $k \in \left\{ 1 , \ldots , l \right\},$ 
on $\B_{g_0} ( x_0 , k )$ over the time interval $t \in \left[0, T_k \right],$ 
still satisfying that $g_l(0) = g_0$ where defined. Moreover,
our repeated applications of Lemma \ref{PEL} provide,
in particular, the estimates
\beq
\label{estimates}
\twopartcond
	{\cR_{g_l(t)} + \al_k \mathcal{I} \in \cC_{\I} } 
	{ \B_{g_0} \left( x_0 , k \right) \times \left[0,T_k\right]}
	{|\Rm|_{g_l(t)} \leq \frac{C_k}{t}} 
	{ \B_{g_0} \left( x_0 , k \right) \times \left(0,T_k\right]}
\eeq
for each $k \in \left\{1 , \ldots , l \right\}$, which completes the proof.
\end{proof}

\appendix

\section{Appendix - Supporting Results}\label{appA}

Here we collect some results 
from \cite{MT18} and \cite{Hoc19}, in slightly modified forms.
The following is a variant of Lemma A.4 from \cite{MT18}, which in turn originates in Lemma 2.3 in \cite{Topping1}. It differs by a parabolic scaling and a reduction in the conclusions.

\begin{lemma}[Volume control]
\label{Volume 2}
Suppose that $\left( M^n , g(t) \right)$ is a smooth Ricci flow over the time interval $t \in \left[0,T\right)$ and that 
for some $R \geq 1$, $\ep>0$ and $x_0 \in M$
we have  
$ \B_{g(0)} ( x_0 , R+\ep) \subset \subset M$.
Moreover assume that  
\begin{itemize}
	\item $\Ric_{g(t)} \geq -K$ on $\B_{g(0)} (x_0 , R+\ep ),$ for some $K > 0$ and all $t \in \left[0,T\right),$
	\item $| \Rm |_{g(t)} \leq \frac{c_0}{t}$ on $\B_{g(0)} (x_0 , R+\ep ),$ for some $c_0 > 0$ and all $t \in \left(0,T\right),$
	\item $\VolBB_{g(0)}  ( x_0 , 1 )  \geq v_0 > 0.$
\end{itemize}
Then there exist $\mu = \mu \left( v_0 , K , R , n, \ep \right) > 0$ and $\hat{T} = \hat{T} \left( v_0 , c_0 , K , n , R, \ep \right) > 0$ 
such that for all 
$t \in \left[0,T\right) \cap [0, \hat{T} )$ 
we have 
$\B_{g(t)} ( x_0 , R ) \subset \B_{g(0)} ( x_0 , R+\ep ),$ 
and 
$\VolBB_{g(t)} ( x_0 , 1 ) \geq \mu.$
\end{lemma}

\noindent
The following results from \cite{Topping1} relate geodesic balls
taken with respect to the metric at different times of a smooth Ricci flow satisfying various local curvature bounds.

\begin{lemma}[The shrinking balls lemma; Corollary 3.3 in \cite{Topping1}]
\label{nested balls}
There exists a constant $\beta = \beta (n) \geq 1$ such that the following is true.
Suppose $M$ is a smooth $n$-manifold and $g(t)$ is a smooth Ricci flow on $M$ defined for all times $0 \leq t \leq T.$  
Suppose $x_0 \in M$ and $r > 0$ are such that
$\B_{g(0)} (x_0 , r) \subset \subset M.$
Further assume that for some $c_0 > 0$ we have
$|\Rm|_{g(t)} \leq \frac{c_0}{t},$
or more generally
$\Ric_{g(t)} \leq \frac{c_0(n-1)}{t},$ 
throughout 
$\B_{g(0)}(x_0,r) \cap \B_{g(t)} ( x_0 , r - \beta \sqrt{c_0 t} )$ for each $t \in (0,T].$ 
Then whenever $0 \leq s \leq t \leq T,$ we have 
\beq
\label{general time} 
	\B_{g(t)}  \left( x_0 , r - \beta \sqrt{c_0 t} \right) \subset \B_{g(s)} \left(x_0 , r- \beta \sqrt{c_0 s}\right). 
\eeq
In particular, for all $0 \leq t \leq T$ 
\beq
\label{initial time} 
	\B_{g(t)}  \left( x_0 , r - \beta \sqrt{c_0 t} \right) \subset \B_{g(0)} (x_0 , r). 
\eeq 

\end{lemma}

\begin{lemma}[The expanding balls lemma; see Lemma 3.1 in \cite{Topping1} and Lemma 2.1 in \cite{Topping2}]
\label{expanding balls}
Let $K , T > 0$ both be given.
Suppose $g(t)$ is a smooth Ricci flow on a smooth $n$-manifold $M$, defined for all times $-T \leq t \leq 0$. 
Let $x_0 \in M$ with $R > 0$ such that
$\B_{g(0)} (x_0 , R) \subset \subset M$ and 
for each $t \in [-T,0]$ suppose that we have 
$\Ric_{g(t)} \geq -K$ throughout 
$\B_{g(0)}(x_0,R) \cap \B_{g(t)} \left( x_0 , Re^{Kt} \right) \subset \B_{g(t)} \left( x_0 , R \right).$ 
Then for all $t \in [-T,0]$ 
\beq
	\label{expanding contain}  
 		\B_{g(t)} \left(x_0 , Re^{Kt}\right)
 		\subset
 		\B_{g(0)} \left(x_0 , R\right). 
\eeq
\end{lemma}
\vskip 4pt
\noindent
Finally we record the following result from \cite{Hoc19} which
establishes the propagation of  lower curvature bounds 
forwards in time under Ricci flows that may be incomplete.
Details of all curvature cones within the following result
may be found in \cite{Wilking}. 


\begin{proposition}[Proposition II.2.6 in \cite{Hoc19}]
\label{Hoc_II.2.6}
Let $n \in \N$ and $c_0 > 0$ both be given. Then there is a constant 
$A = A ( n , c_0 ) > 0$ for which the following is true.
Assume $\cC$ is one of the invariant curvature cones 
$\cC_{CO}, \cC_{\I}$ or $\cC_{CSC}$
that are described in Remark \ref{other_cones}. 
Let $\left( M , g(t) \right)$ be a smooth $n$-dimensional Ricci flow, 
defined for $0 \leq t \leq T,$ satisfying 
$\Sc_{g(t)} \geq -1$ throughout
$M \times [0,T]$, and both $|\Rm|_{g(t)} \leq \frac{c_0}{t}$ and $\inj_{g(t)} \geq \sqrt{\frac{t}{c_0}}$ throughout $M \times (0,T].$
Then, if 
$\cR_{g(0)} + \mathcal{I} \in \mathcal{C}$ throughout $M$, 
we may conclude that 
$\cR_{g(t)} + A \rho_0^{-2} \mathcal{I} \in \mathcal{C}$
throughout $M \times [0,T]$,
where $\rho_0 : M \to [0,1]$ is defined by
$\rho_0 (x) := \sup \left\{ r \in (0,1] : \B_{g(0)} (x,r) \subset \subset M \right\}.$
\end{proposition}

\footnotesize

\noindent
AM:
\emph{
a.mcleod@ucl.ac.uk} 

\noindent
\url{https://iris.ucl.ac.uk/iris/browse/profile?upi=AMCLE50}

\noindent
{\sc Department of Mathematics, University College London, London,
WC1H 0AY, UK} 

\vskip 5pt

\noindent
PT:\\
\url{http://homepages.warwick.ac.uk/~maseq/}

\noindent
{\sc Mathematics Institute, University of Warwick, Coventry,
CV4 7AL, UK}

\end{document}